\documentclass[11pt]{amsart}

\usepackage{amssymb,amsmath,color,hyperref}
\usepackage[mathscr]{eucal}

\theoremstyle{plain}
\newtheorem{thm}{Theorem}[section]
\newtheorem{theorem}[thm]{Theorem}
\newtheorem{lem}[thm]{Lemma}

\newtheorem{prop}[thm]{Proposition}

\theoremstyle{definition}

\newtheorem{ex}[thm]{Example}

\theoremstyle{remark}

\newtheorem{remark}{Remark}
\newtheorem*{remark*}{Remark}

\numberwithin{equation}{section}

        \newcommand{\field}[1]{{\mathbb{#1}}}
        \newcommand{\NN}{\field{N}}
        \newcommand{\ZZ}{\field{Z}}
        
        \newcommand{\RR}{\field{R}}

\begin{document}

\title[Eigenvalue distribution in gaps of the essential spectrum]{Eigenvalue distribution in gaps of the essential spectrum of the Bochner-Schr\"odinger operator}

\author[Y. A. Kordyukov]{Yuri A. Kordyukov}
\address{Institute of Mathematics, Ufa Federal Research Centre, Russian Academy of Sciences, 112~Chernyshevsky str., 450008 Ufa, Russia} \email{yurikor@matem.anrb.ru}

%
%
\begin{abstract}
The Bochner-Schr\"odinger operator $$H_{p}=\frac 1p\Delta^{L^p}+V$$ on high tensor powers $L^p$ of a Hermitian line bundle $L$ on a Riemannian manifold $X$ of bounded geometry is studied under the assumption of non-degeneracy of the curvature form of $L$. For large $p$, the spectrum of $H_p$ asymptotically coincides with the union of all local Landau levels of the operator at the points of $X$. Moreover, if the union of the local Landau levels over the complement of a compact subset of $X$ has a gap, then the spectrum of $H_{p}$ in the gap is discrete. The main result of the paper is the trace asymptotics formula associated with these eigenvalues. As a consequence, we obtain a Weyl type asymptotic formula for the eigenvalue counting function.
\end{abstract}


 \maketitle
\section{Preliminaries and main results}
Let $(X,g)$ be a smooth Riemannian manifold of dimension $d$ without boundary, $(L,h^L)$ a Hermitian line bundle on $X$ with a Hermitian connection $\nabla^L$. We suppose that $(X, g)$ is  a manifold of bounded geometry and $L$ has bounded geometry. This means that the curvatures $R^{TX}$ and $R^L$ of the Levi-Civita connection $\nabla^{TX}$ and the connection $\nabla^L$, respectively, and their derivatives of any order are uniformly bounded on $X$ in norm induced by $g$ and $h^L$, and the injectivity radius $r_X$ of $(X, g)$ is positive. One can also introduce an additional twisting by a Hermitian vector bundle $(E,h^E)$  with a Hermitian connection $\nabla^E$ of bounded geometry, but we will omit it for simplicity. 

For any $p\in \NN$, let $L^p:=L^{\otimes p}$ be the $p$th tensor power of $L$ and let
\[
\nabla^{L^p}: {C}^\infty(X,L^p)\to
{C}^\infty(X, T^*X \otimes L^p)
\] 
be the Hermitian connection on $L^p$ induced by $\nabla^{L}$. Consider the induced Bochner Laplacian $\Delta^{L^p}$ acting on $C^\infty(X,L^p)$ by
\begin{equation}\label{e:def-Bochner}
\Delta^{L^p}=\big(\nabla^{L^p}\big)^{\!*}\,
\nabla^{L^p},
\end{equation} 
where $$\big(\nabla^{L^p}\big)^{\!*}: {C}^\infty(X,T^*X\otimes L^p)\to
{C}^\infty(X,L^p)$$ is the formal adjoint of  $\nabla^{L^p}$. Let $V\in C^\infty(X)$ be a real-valued function on $X$. We assume that $V$ and its derivatives of any order are uniformly bounded on $X$ in norm induced by $g$. 
We study the Bochner-Schr\"odinger operator $H_p$ acting on $C^\infty(X,L^p)$ by
\[
H_{p}=\frac 1p\Delta^{L^p}+V. 
\] 
The operator $H_p$ is self-adjoint in the Hilbert space $L^2(X,L^p)$ with domain  $H^2(X, L^p)$, the second Sobolev space. We are interested in asymptotic spectral properties of the operator $H_p$ in the semiclassical limit $p\to +\infty$.  

Consider the real-valued closed 2-form $\mathbf B$ (the magnetic field) given by 
\begin{equation}\label{e:def-omega}
\mathbf B=iR^L. 
\end{equation} 
We assume that $\mathbf B$ is non-degenerate. Thus, $X$ is a symplectic manifold. In particular, its dimension is even, $d=2n$, $n\in \NN$. 

For $x\in X$, let $B_x : T_xX\to T_xX$ be the skew-adjoint operator such that 
\[
\mathbf B_x(u,v)=g(B_xu,v), \quad u,v\in T_xX. 
\]
The operator $|B_x|:=(B_x^*B_x)^{1/2} : T_xX\to T_xX$ is a positive self-adjoint operator. We assume that it is uniformly positive on $X$: 
\begin{equation}\label{e:uniform-positive} 
b_0:=\inf_{x\in X}|B_x|>0.
\end{equation}


\begin{remark}
 Assume that the Hermitian line bundle $(L,h^L)$ is trivial. Then we can write $\nabla^L=d-i \mathbf A$ with a real-valued 1-form $\mathbf A$ (the magnetic potential), and we have
\[
R^L=-id\mathbf A,\quad \mathbf B=d\mathbf A. 
\]
We may assume that $p\in \mathbb R$. The operator $H_p$ is given by
\[
H_p=\frac{1}{p}(d-ip\mathbf A)^*(d-ip\mathbf A)+V.
\]
It has the form of the semiclassical magnetic Schr\"odinger operator:
\[
H_p=\frac{1}{\hbar}(i\hbar d+\mathbf A)^*(i\hbar d+\mathbf A)+V, \quad \hbar=\frac{1}{p},\quad p\in \RR. 
\]
\end{remark} 

\begin{remark} 
If $X$ is the linear space $\RR^{2n}$ with coordinates $Z=(Z_1,\ldots,Z_{2n})$, we can write the 1-form $\bf A$ as
\[
{\bf A}= \sum_{j=1}^{2n}A_j(Z)\,dZ_j,
\]
the matrix of the Riemannian metric $g$ as $g(Z)=(g_{j\ell}(Z))_{1\leq j,\ell\leq 2n}$
and its inverse as $$g(Z)^{-1}=(g^{j\ell}(Z))_{1\leq j,\ell\leq 2n}.$$
Denote $|g(Z)|=\det(g(Z))$. Then $\bf B$ is given by 
\[
{\bf B}=\sum_{j<k}B_{jk}\,dZ_j\wedge dZ_k, \quad
B_{jk}=\frac{\partial A_k}{\partial Z_j}-\frac{\partial
A_j}{\partial Z_k}.
\]
Moreover, the operator $H_p$ has the form
\[
H_p=\frac{1}{p}\frac{1}{\sqrt{|g|}}\sum_{1\leq j,\ell\leq 2n}\left(i \frac{\partial}{\partial Z_j}+pA_j\right) \left[\sqrt{|g|} g^{j\ell} \left(i \frac{\partial}{\partial Z_\ell}+pA_\ell\right)\right]+V.
\]
Our assumptions hold, if the matrix $(B_{j\ell}(Z))$ has full rank $2n$ and its eigenvalues are separated from zero uniformly on $Z\in \RR^{2n}$, for any $\alpha \in \ZZ^{2n}_+$ and $1\leq j,\ell\leq 2n$, we have 
\[
\sup_{Z\in \RR^{2n}}|\partial^\alpha g_{j\ell}(Z)|<\infty, \quad \sup_{Z\in \RR^{2n}}|\partial^\alpha B_{j\ell}(Z)|<\infty, 
\]
and the matrix $(g_{j\ell}(Z))$ is positive definite uniformly on $Z\in \RR^{2n}$.

In particular, when $g$ is the standard Euclidean metric in $\mathbb R^{2n}$, the operator $H_p$ takes the form (with $\hbar=1/p$, $p\in \mathbb R$):
\[
H_\hbar=\frac{1}{\hbar}\sum_{j=1}^{2n}\left(i\hbar \frac{\partial}{\partial Z_j}+A_j\right)^2+V(Z).
\]
It can be rewritten as 
 \[
H_\hbar=\sum_{j=1}^{2n}\left(i\hbar^{1/2}\frac{\partial}{\partial Z_j}+\hbar^{-1/2}A_j\right)^2+V(Z).
\]
and treated as the two-parameter magnetic Schr\"odinger operator 
\begin{equation}\label{e:Hmu}
H_{h,\mu}=\sum_{j=1}^{2n}\left(ih\frac{\partial}{\partial Z_j}+\mu A_j\right)^2+V(Z)
\end{equation}
in the asymptotic regime $h\to 0$, $h\mu=1$. 

In the case of a homogeneous magnetic field
\[
{\bf A}= \sum_{k=1}^{n}\left(\frac{a_k}{2}Z_{2k-1}dZ_{2k}-\frac{a_k}{2}Z_{2k}dZ_{2k-1}\right),\quad {\bf B}= \sum_{k=1}^{n} a_kdZ_{2k-1}\wedge dZ_{2k}, 
\]
the operator $H_\hbar$ has the form
\[
H_\hbar=\frac{1}{\hbar}\sum_{k=1}^{n}\left[\left(i\hbar \frac{\partial}{\partial Z_{2k-1}}-\frac{a_k}{2}Z_{2k}\right)^2+\left(i\hbar \frac{\partial}{\partial Z_{2k}}+\frac{a_k}{2}Z_{2k-1}\right)^2\right]+V(Z).
\]
It is unitarily equivalent to the magnetic Schr\"odinger operator with expanding electric potential
\[
\hat H_\hbar=\sum_{k=1}^{n}\left[\left(i\frac{\partial}{\partial Z_{2k-1}}-\frac{a_k}{2}Z_{2k}\right)^2+\left(i\frac{\partial}{\partial Z_{2k}}+\frac{a_k}{2}Z_{2k-1}\right)^2\right]+V(\hbar^{1/2}Z).
\] 
In particular, if $V\equiv 0$, the semiclassical limit for the Dirichlet realization of the operator $H_\hbar$ in a bounded domain of $\mathbb R^{2n}$ is equivalent to the thermodynamic limit of a 2D Fermi gas submitted to a constant external magnetic field (see, for instance, \cite{CFFH} and the references therein).
 \end{remark}
 
%

For an arbitrary $x\in X$, the operator $B_{x}$ is skew-adjoint. Therefore, its eigenvalues have the form $\pm i a_j(x), j=1,\ldots,n,$ with $a_j(x)>0$. By \eqref{e:uniform-positive}, we have 
$$
a_j(x)\geq b_0>0, \quad x\in X, \quad j=1,\ldots,n.
$$   
Set
\[
\Lambda_{\mathbf k}(x)=\sum_{j=1}^n(2k_j+1) a_j(x)+V(x), \quad \mathbf k=(k_1,\cdots,k_n)\in\ZZ_+^n, \quad x\in X.
\]
%
For any $x\in X$, the set
 \[
 \Sigma_{x}:=\left\{\Lambda_{\mathbf k}({x})\,:\, \mathbf k\in\ZZ_+^n \right\}
\]
is the spectrum of the magnetic Schr\"odinger operator in $T_{x}X$ with constant magnetic field $\mathbf B_{x}$ and constant electric potential $V(x)$ (see, for instance, \cite{Ko22} for more details). Therefore, the numbers $\Lambda_{\mathbf k}(x)$ can be naturally called local Landau levels at $x$. Let
\[
\Sigma:=\bigcup_{x\in X}\Sigma_{x}. 
\] 

\begin{theorem}[\cite{Ko22}]\label{t:spectrum}
For any $K>0$, there exists a $c>0$ such that for any $p\in \NN$ the spectrum of $H_{p}$ in the interval  $[0,K]$  is  contained in the $cp^{-1/4}$-neighborhood of $\Sigma$.
\end{theorem}

When $X$ is compact, a stronger result, with $p^{-1/2}$ instead of $p^{-1/4}$ was proved by L. Charles \cite{charles21}. This estimate seems to be optimal. 

For an interval $[a,b]$, let $\mathcal K_{[a,b]}$ be the closed subset of $X$ given by 
\[
\mathcal K_{[a,b]}=\{x\in X \,:\,  \Sigma_x\cap [a,b]\neq\emptyset\}.
\]
In other words, $x\in \mathcal K_{[a,b]}$ iff $\Lambda_{\mathbf k}(x)\in [a,b]$ for some $\mathbf k \in \mathbb Z^n_+$.

\begin{theorem}[\cite{essential},Theorem 1.2]\label{t:ess-spectrum}
Assume that, for an interval $[a,b]\subset \mathbb R$, the set $\mathcal K_{[a,b]}$
is compact. Then there exists an $\epsilon>0$ such that for any $p\in \NN$ the spectrum of $H_{p}$ in $[a+\epsilon p^{-1/4},b-\epsilon p^{-1/4}]$ is discrete.  
\end{theorem}

As in Theorem~\ref{t:spectrum}, the order $p^{-1/4}$ doesn't seem to be optimal and, probably, can be improved.

The main result of the paper is the following theorem. 

\begin{thm}\label{t:trace} Under the assumptions of Theorem \ref{t:ess-spectrum}, there exists a sequence of distributions $f_r \in \mathcal D^\prime(\mathbb R), r=0,1,\ldots$, such that for any $\varphi\in C^\infty_c(\mathbb R)$ supported in $(a,b)$, we have an asymptotic expansion
\begin{equation}\label{e:trace}
\operatorname{tr} \varphi(H_{p})\sim p^{n}\sum_{r=0}^{\infty}\langle f_{r}, \varphi\rangle p^{-\frac{r}{2}}, \quad p \to \infty.
\end{equation}
\end{thm}

The leading coefficient in \eqref{e:trace} is given by 
\begin{equation}\label{e:f0a-2n2}
\langle f_0, \varphi\rangle=\frac{1}{(2\pi)^{n}}\sum_{\mathbf k\in\ZZ_+^n} \int_X \varphi(\Lambda_{\mathbf k}(x_0))d\mu(x_0),
\end{equation}
where $d\mu=\mathbf B^n/n!$ is the Liouville volume form of the symplectic manifold $(X,\mathbf B)$.
We have
\[
d\mu=\left(\prod_{j=1}^n a_j(x_0)\right) dv_g(x_0),
\]
where $dv_g$ denotes the Riemannian volume form on $X$ associated with $g$.

The next coefficients have the form
\begin{equation}\label{e:fra-2n2}
\langle f_{r}, \varphi\rangle=\sum_{\mathbf k\in\ZZ_+^n} \sum_{\ell=0}^{m}\int_X P_{\mathbf k,\ell,r}(x_0) \varphi^{(\ell)}(\Lambda_{\mathbf k}(x_0))d\mu(x_0),\quad r\geq 1,
\end{equation}
where $P_{\mathbf k,\ell,r}$ is polynomially bounded in $\mathbf k$. 

Note that the sums over $\mathbf k$ in \eqref{e:f0a-2n2} and \eqref{e:fra-2n2} have finitely many nonzero terms, because $\varphi$ is compactly supported and $\Lambda_{\mathbf k}(x_0)\to +\infty$ as $\mathbf k\to \infty$ for any $x_0\in X$. 

As a standard corollary (see, for instance, \cite[Corollary 9.7]{DS99}), we derive Weyl law for the counting function of eigenvalues of $H_p$ in $(a,b)$. For any self-adjoint operator $A$ with discrete spectrum on an interval $I\subset \mathbb R$, we will denote by $N(I;A)$ the number of its eigenvalues in $I$, counting   multiplicities. 

Let $\mathcal S$ be the at most countable set of all $\alpha\in (a,b)$ such that 
\[
\lim_{\varepsilon\to 0+}\sum_{\mathbf k\in\ZZ_+^n} \mu(\{x\in X : \Lambda_{\mathbf k}({x}) \in [\alpha-\varepsilon,\alpha+\varepsilon]\})\neq 0. 
\]

\begin{thm}\label{t:Weyl}
Under the assumptions of Theorem \ref{t:ess-spectrum}, for any $[\alpha,\beta]\subset (a,b)$ with $\alpha,\beta\not\in \mathcal S$, we have
\begin{equation}\label{e:Demailly2}
N([\alpha,\beta];H_p)\sim\frac{p^{n}}{(2\pi)^n}\sum_{\mathbf k\in\ZZ_+^n} \mu(\{x\in X : \Lambda_{\mathbf k}({x}) \in [\alpha,\beta] \}), \quad p\to +\infty.
\end{equation}
\end{thm}

\begin{ex}\label{ex:b}
Consider the case when $d=2$ and $V\equiv 0$. Then the magnetic two-form $\bf B$ is a volume form on $X$ and therefore can be identified with the function $b\in C^\infty_b(X)$
given by
\[
{\bf B}=b\,dv_g. 
\]
Denote $b_{min}=\inf_{x\in X} b(x)>0$ and $b_{max}=\sup_{x\in X} b(x)$. Then if 
\[
(2k-1)b_{max}<(2k+1)b_{min}
\] 
for some $k=1,2,\ldots,$ then 
$$
\left((2k-1)b_{max}, (2k+1)b_{min}\right)\neq \emptyset\ \text{and}\ \left((2k-1)b_{max}, (2k+1)b_{min}\right)\cap \Sigma= \emptyset.
$$ 
By Theorem~\ref{t:spectrum}, for any $$(\alpha,\beta)\in ((2k-1)b_{max}, (2k+1)b_{min}),$$ there exists $p_0\in \mathbb N$ such that  
\[
(\alpha,\beta)\cap \sigma(H_p)= \emptyset, \quad p>p_0. 
\]
Denote $b_{K,min}=\inf_{x\in X \setminus K} b(x)>0$ and $b_{K,max}=\sup_{x\in X \setminus K} b(x)$ for some compact $K\subset X$. Then if 
\[
(2k-1)b_{K,max}<(2k+1)b_{K,min}
\] 
for some $k=1,2,\ldots,$ then, by Theorem~\ref{t:ess-spectrum}, for any $$(\alpha,\beta)\in ((2k-1)b_{K,max}, (2k+1)b_{K,min}),$$ there exists $p_0\in \mathbb N$ such that  
\[
(\alpha,\beta)\cap \sigma_{ess}(H_p)= \emptyset, \quad p>p_0. 
\]
By Theorem \ref{t:trace}, for any $\varphi\in C^\infty_c(\mathbb R)$ supported in $(\alpha,\beta)$, we have an asymptotic expansion \eqref{e:trace}:
\[
\operatorname{tr} \varphi(H_{p})\sim \langle f_0, \varphi\rangle p+\langle f_1, \varphi\rangle p^{\frac 12}+ \langle f_{2}, \varphi\rangle +\ldots , \quad p \to \infty,
\]
with the leading coefficient given by
\[
\langle f_0, \varphi\rangle=\frac{1}{2\pi}\int_X b(x_0) \varphi((2k+1)b(x_0)) dv_g(x_0).
\]
Finally, the formula \eqref{e:Demailly2} reads as:
\[
N([\alpha,\beta];H_p)\sim \frac{p}{2\pi} \int_{\{x\in X : (2k+1)b(x) \in [\alpha,\beta]\}}b(x) dv_g(x) 
, \quad p\to +\infty.
\]
\end{ex}

\begin{ex}\label{ex:V}
Consider the case when $d=2$, the magnetic field is constant: $b(x)\equiv b, x\in X$, and the electric potential $V$ satisfies
\[
\lim_{x\to \infty}V(x)=0. 
\]
Then the essential spectrum of $H_p$ coincides with  
\[
\Sigma_0=\{(2k+1)b : k\in \mathbb Z_+\}. 
\]

The trace formula \eqref{e:trace} holds for any $\varphi\in C^\infty_c(\mathbb R)$ supported in $\mathbb R\setminus \Sigma_0$ with the leading coefficient given by 
\[
\langle f_0, \varphi\rangle=\frac{b}{2\pi}\sum_{k=0}^\infty \int_X \varphi((2k+1)b+V(x))dv_g(x).
\]

Finally, the formula \eqref{e:Demailly2} reads as:
\begin{multline*}
N([\alpha,\beta];H_p)\sim\frac{pb}{2\pi}\sum_{k=0}^\infty {\rm vol}_g\left(\{x\in X : (2k+1)b+V(x) \in [\alpha,\beta]\}\right) \\
=\frac{pb}{2\pi}\sum_{k=0}^\infty {\rm vol}_g (V^{-1}[\alpha-(2k+1)b,\beta-(2k+1)b])
, \quad p\to +\infty.
\end{multline*}
\end{ex}


When $X$ is a compact Riemannian manifold and, therefore, the spectrum of $H_p$ is discrete, the trace formula \eqref{e:trace} is proved in \cite{Bochner-trace}. When the Hermitian line bundle $(L,h^L)$ is trivial, it is a particular case of the Gutzwiller trace formula for the semiclassical magnetic Schr\"odinger operator and the zero energy level of the classical Hamiltonian, which is critical in this case (see the survey \cite{trace-zero} for more details).

When $X$ is the Euclidean space $\mathbb R^{2n}$ equipped with the standard Riemannian metric, the case described in Example~\ref{ex:V} was studied in \cite{DD14}. The authors prove a complete asymptotic expansion in powers of $h$ and a Weyl-type asymptotic formula with optimal remainder estimate for the counting function of eigenvalues. 

The formula for the leading coefficient in the trace formula is known for a long time (see, for instance, \cite{CdV86,ELS97,LSY94,Sob94,Sob95, Sob98,T87,T97} for some related results). The asymptotic formula for the counting function of the discrete spectrum of the two-dimensional Landau Hamiltonian perturbed by a nonsmooth expanding potential $V\in L^1(\mathbb R^2)\cap L^2(\mathbb R^2)$ was derived in \cite{RS} (see \cite{P09} for another proof). For the Dirichlet realization of the operator $H_\hbar$ in a bounded domain of $\mathbb R^{2n}$, the second coefficient in the trace formula, which describes a contribution of the boundary, was studied in \cite{CFFH}.  
The asymptotic estimates with remainder for the local spectral function of the two-parameter magnetic Schr\"odinger operator $H_{h,\mu}$ given by \eqref{e:Hmu} in the asymptotic regime $h\to 0$, $h\mu=1$ (called the case of intermediate magnetic field) were studied by V. Ivrii (see \cite[Section 13.6]{{I19III}} for two-dimensional case and \cite[Section 19.6]{{I19IV}} for higher dimensional case).

We also mention that G. Raikov \cite{R98} proved a Weyl formula for the counting function of the operator
\[
\mathcal H_p=\frac{1}{p}(d-ip\mathbf A)^*(d-ip\mathbf A)+\frac 1pV=\frac{1}{\hbar}(i\hbar d+\mathbf A)^*(i\hbar d+\mathbf A)+\hbar V
\]
in $\mathbb R^{2n}$ with constant magnetic field of full rank and decreasing electric potential $V$ in a shrinking interval, approaching a fixed Landau level: $(\Lambda_{\mathbf k} + \frac{1}{p}\lambda_1,\Lambda_{\mathbf k} + \frac{1}{p}\lambda_2)$ with some $\mathbf k\in \mathbb Z_+^n$ and $\lambda_1 < \lambda_2$, $\lambda_1\lambda_2 > 0$. M. Dimassi \cite{D01} improved this result with a sharp remainder estimate and gave a complete asymptotic expansion of the trace formula. 

A recent paper \cite{Y20} studies asymptotics of individual  eigenvalues in the gaps of the two dimensional Schr\"odinger operator with strong magnetic field.

 \section{Proof of the main theorem} \label{s:proof}

This section is devoted to the proof of Theorem~\ref{t:trace}. 

For $\varphi\in C^\infty_c(\mathbb R)$, let $K_{\varphi(H_{p})}\in {C}^{\infty}(X\times X)$ denote the Schwartz kernel of the operator $\varphi(H_{p})$ with respect to the Riemannian volume form $dv_g$.  
Here the starting point is the local trace formula proved in \cite[Corollary 1.2]{Bochner-trace}. It states that, for any $\varphi\in C^\infty_c(\mathbb R)$, the following complete asymptotic expansion holds uniformly on $x_0\in X$:
\begin{equation}\label{e:local-trace}
p^{-n}K_{\varphi(H_{p})} (x_0,x_0)\sim \sum_{r=0}^{+\infty}f_{r}(x_0)p^{-\frac{r}{2}},\quad x_0\in X, \quad p\to +\infty,
\end{equation}
with some smooth functions $f_{r},$ $r=0,1,\ldots$, on $X$ which means that, for any $N\in \mathbb N$, there exists $C>0$ such that for any $p\geq 1$ and $x_0\in X$,  
\[
\left|p^{-n}K_{\varphi(H_{p})} (x_0,x_0)-\sum_{r=0}^Nf_{r}(x_0)p^{-\frac{r}{2}}\right|\leq Cp^{-\frac{N+1}{2}}.
\]
The leading coefficient in \eqref{e:local-trace} is given by 
\begin{equation}\label{e:f0}
f_0(x_0)=\frac{1}{(2\pi)^{n}}\left(\prod_{j=1}^n a_j(x_0)\right) \sum_{\mathbf k\in\ZZ_+^n} \varphi(\Lambda_{\mathbf k}(x_0)).
\end{equation}
The next coefficients have the form
\begin{equation}\label{e:fr}
f_{r}(x_0)=\left(\prod_{j=1}^n a_j(x_0)\right) \sum_{\mathbf k\in\ZZ_+^n} \sum_{\ell=0}^{m} P_{\mathbf k,\ell,r}(x_0) \varphi^{(\ell)}(\Lambda_{\mathbf k}(x_0)),\quad r\geq 1,
\end{equation}
where $P_{\mathbf k,\ell,r}$ is polynomially bounded in $\mathbf k$. 

Now suppose that $\varphi$ is supported in $(a,b)$ and $[\alpha,\beta]\subset (a,b)$ is an interval such that ${\rm supp}\,\varphi\subset (\alpha,\beta)$. In particular, it follows that the coefficients $f_r$ are supported in $\mathcal K_{[\alpha,\beta]}$.

For brevity, denote $N_p=N([\alpha,\beta];H_p)$. Let $\lambda_{p,j}, j=1,2,\ldots,N_p,$ be the eigenvalues of $H_p$ in $[\alpha,\beta]$ taking into account the multiplicities and $\{u_{p,j}\in C^\infty(X,L^p), j=1,2,\ldots, N_p\}$ the corresponding orthonormal system of eigensections:
\begin{equation}\label{e:eigen}
H_pu_{p,j} = \lambda_{p,j}u_{p,j}, \quad p\in \mathbb N, \quad j=1,2,\ldots, N_p. 
\end{equation}

By [\cite{essential}, Theorem 1.3], there exist $p_0\in \mathbb N$ and $C, c>0$ such that, for any $p>p_0$ and $j=1,2,\ldots, N_p$, we have  
\begin{equation}\label{e:exp1}
\int_X e^{2c\sqrt{p} d (x,\mathcal K_{[\alpha,\beta]})}|u_{p,j}(x)|^2dv_g(x) \leq C.
\end{equation}

We also need a rough upper estimate for the eigenvalue distribution function $N([\alpha,\beta];H_p)$.

\begin{prop}\label{p:upperbound}
For any $[\alpha,\beta]\subset (a,b)$, there exists $C>0$ such that
\[
N([\alpha,\beta];H_p)\leq Cp^n, \quad p\in \mathbb N.
\]
\end{prop}

The proof of Proposition~\ref{p:upperbound} will be given later. First, we complete the proof of Theorem~\ref{t:trace}.

We have
\[
K_{\varphi(H_{p})}(x_0,x_0)=\sum_{j=1}^{N_p} \varphi(\lambda_{p,j})|u_{p,j}(x_0)|^2. 
\]
By \eqref{e:exp1} and Proposition~\ref{p:upperbound}, we infer that
\begin{multline}\label{e:weight}
\int_X e^{2c\sqrt{p} d (x_0,\mathcal K_{[\alpha,\beta]})} K_{\varphi(H_{p})}(x_0,x_0)dv_g(x_0)\\ =\sum_{j=1}^{N_p} \varphi(\lambda_{p,j})\int_X e^{2c\sqrt{p} d (x_0,\mathcal K_{[\alpha,\beta]})} |u_{p,j}(x_0)|^2dv_g(x_0)\\ \leq C\sup_{\lambda\in \mathbb R} |\varphi(\lambda)|N_p\leq C_1p^n. 
\end{multline}

\begin{lem}\label{l:drho}
For any $[\alpha,\beta]\subset [a,b]$, there exists an $r>0$ such that
\[
d (x_0,\mathcal K_{[\alpha,\beta]})>r
\]
for any $x_0\in X\setminus \mathcal K_{[a,b]}$.
\end{lem}

\begin{proof}
Take an arbitrary $x\in \mathcal K_{[\alpha,\beta]}$. By definition, $\Lambda_{\mathbf k^\prime}(x)\in [\alpha,\beta]$ for some $\mathbf k^\prime \in \mathbb Z^n_+$. For any such $\mathbf k^\prime$, we have
\[
\beta\geq \Lambda_{\mathbf k^\prime}(x) = \sum_{j=1}^n (2k^\prime_j+1) a_j(x)+V(x)\geq (2|\mathbf k^\prime|+n)b_0+\inf_{x\in X}V(x),
\]
that provides the upper bound
\[
|\mathbf k^\prime|:=\sum_{j=1}^n k^\prime_j\leq \frac{1}{2b_0}\left(\beta-\inf_{x\in X} V(x)\right)-\frac n2. 
\]
By the mini-max principle and the fact that $B$ and $V$ are in $C^\infty_b$, for any $x\in X$ and $x_0\in X$, we have 
\[
|a_j(x)-a_j(x_0)|\leq |B_x-B_{x_0}|\leq C_1d(x,x_0), \quad j=1,2,\ldots,n. 
\]
\[
|V(x)-V(x_0)| \leq C_2d(x,x_0),
\]
where $C_1>0$ and $C_2>0$ are independent of $x_0$ and $x$. 

Therefore, for any $x\in \mathcal K_{[\alpha,\beta]}$ and $x_0\in X$, we have
\[
|\Lambda_{\mathbf k^\prime}(x)-\Lambda_{\mathbf k^\prime}(x_0)|
\leq \sum_{j=1}^n(2k^\prime_j+1)|a_j(x)-a_j(x_0)|+|V(x)-V(x_0)|\leq Cd(x,x_0), 
\]
where $C>0$ is independent of $x_0$ and $x$. 

If $x_0\not\in \mathcal K_{[a,b]}$ iff $\Lambda_{\mathbf k}(x_0)\not \in [a,b]$ for any $\mathbf k \in \mathbb Z^n_+$. In particular, $\Lambda_{\mathbf k^\prime}(x_0)\not \in [a,b]$. It follows that
\[
|\Lambda_{\mathbf k^\prime}(x)-\Lambda_{\mathbf k^\prime}(x_0)|\geq \min(\alpha-a, b-\beta),
\]
and the claim is immediate. 
\end{proof}

By Lemma \ref{l:drho} and \eqref{e:weight}, we get
\begin{multline*}
\int_{X\setminus \mathcal K_{[a,b]}}K_{\varphi(H_{p})}(x_0,x_0)dv_g(x_0)\\ \leq e^{-2cr\sqrt{p}} \int_{X\setminus \mathcal K_{[a,b]}} e^{2c\sqrt{p} d (x_0,\mathcal K_{[\alpha,\beta]})} K_{\varphi(H_{p})}(x_0,x_0)dv_g(x_0)\leq C_1p^n e^{-2cr\sqrt{p}}.
\end{multline*}
Thus, we see that
\begin{multline*}
p^{-n}\operatorname{tr} \varphi(H_{p})=p^{-n} \int_{X}K_{\varphi(H_{p})}(x_0,x_0)dv_g(x_0)\\ =p^{-n} \int_{\mathcal K_{[a,b]}}K_{\varphi(H_{p})}(x_0,x_0)dv_g(x_0)+p^{-n}\int_{X\setminus \mathcal K_{[a,b]}}K_{\varphi(H_{p})}(x_0,x_0)dv_g(x_0)\\ = p^{-n} \int_{\mathcal K_{[a,b]}}K_{\varphi(H_{p})}(x_0,x_0)dv_g(x_0) + \mathcal O(e^{-2cr\sqrt{p}}).
\end{multline*}
Using the local trace formula \eqref{e:local-trace} and formulae \eqref{e:f0} and \eqref{e:fr}, we complete the proof of Theorem~\ref{t:trace}. 

\begin{proof}[Proof of Proposition~\ref{p:upperbound}]
Recall that $\lambda_{p,j}, j=1,2,\ldots,N_p,$ are the eigenvalues of $H_p$ in $[\alpha,\beta]$ taking into account the multiplicities and $\{u_{p,j}\in C^\infty(X,L^p), j=1,2,\ldots, N_p\}$ is the corresponding orthonormal system of eigensections. Take relatively compact regular domains $D$ and $D_1$ in $X$ such that $\overline{D_1}\subset D$ and $\mathcal K_{[\alpha,\beta]}\subset D_1$. Let $\chi\in C^\infty_c(X)$ satisfy $0\leq \chi\leq 1$, $\chi\equiv 1$ on $D_1$ and ${\rm supp}\,\chi\subset D$. Set
\[
v_{p,j}=\chi u_{p,j}\in C^\infty(X,L^p),\quad p\in \mathbb N, \quad j=1,2,\ldots, N_p. 
\]
It is clear that
\begin{equation}\label{e:vjvj}
\|v_{p,j}\|\leq 1, \quad p\in \mathbb N, \quad j=1,2,\ldots, N_p.
\end{equation}

There exists $r>0$ such that for any $x\in {X\setminus D_1}$ we have $d (x,\mathcal K_{[\alpha,\beta]})>r$. By \eqref{e:exp1}, it follows that 
\begin{multline}\label{e:X-D1}
\int_{X\setminus D_1} |u_{p,j}(x)|^2dv_g(x)\\ \leq e^{-2cr\sqrt{p}} \int_{X\setminus D_1} e^{2c\sqrt{p} d (x,\mathcal K_{[\alpha,\beta]})} |u_{p,j}(x)|^2dv_g(x)\\ \leq C e^{-2cr\sqrt{p}}, \quad p\in \mathbb N, \quad j=1,2,\ldots, N_p.
\end{multline}
Therefore, we obtain
\begin{multline}\label{e:vjvj1}
\|v_{p,j}\|^2\geq \int_{D_1} |u_{p,j}(x)|^2dv_g(x)
=1-\int_{X\setminus D_1} |u_{p,j}(x)|^2dv_g(x)\\ \geq 1-C e^{-2cr\sqrt{p}}, \quad p\in \mathbb N, \quad j=1,2,\ldots, N_p.
\end{multline}

For $j\neq k$, we have
\begin{align*}
(v_{p,j},v_{p,k})& =\int_{X}\chi^2(x) u_{p,j}(x)\overline{u_{p,k}(x)}dv_g(x)
\\
& =\int_{X}(\chi^2(x)-1) u_{p,j}(x)\overline{u_{p,k}(x)}dv_g(x)
\\
& =\int_{X\setminus D_1}(\chi^2(x)-1) u_{p,j}(x)\overline{u_{p,k}(x)}dv_g(x).
\end{align*}
By \eqref{e:X-D1}, we get
\begin{multline}\label{e:vjvk}
|(v_{p,j},v_{p,k})|\\ \leq \left(\int_{X\setminus D_1} |u_{p,j}(x)|^2dv_g(x)\right)^{1/2}\left(\int_{X\setminus D_1} |u_{p,k}(x)|^2dv_g(x)\right)^{1/2}\\ \leq C e^{-2cr\sqrt{p}}, \quad p\in \mathbb N, \quad j=1,2,\ldots, N_p.
\end{multline}

Now we estimate the quadratic form of $H_p$. We can write
\begin{equation}\label{e:Hvjvj0}
(H_pv_{p,j},v_{p,j})=([H_p,\chi] u_{p,j},v_{p,j})+\lambda_{p,j}\|v_{p,j}\|^2.
\end{equation}
It is easy to check that
\begin{equation}\label{e:comm1}
[H_p, \chi]=\frac 1p\left(-2\nabla\chi\cdot\nabla^{L^p} +\Delta\chi\right).
\end{equation}
To find an upper bound for $\|\nabla^{L^p}u_{p,j}\|$, we take an inner product of both parts of \eqref{e:eigen} with $u_{p.j}$:
\[
(H_pu_{p,j},u_{p,j})=\frac 1p\|\nabla^{L^p} u_{p,j}\|^2+(Vu_{p,j},u_{p,j})=\lambda_{p,j}.
\]
Since $\lambda_{p,j}\leq \beta$ and $V$ is bounded on $X$, we obtain
\begin{equation}\label{e:Hvjvj2}
\|\nabla^{L^p} u_{p,j}\|=\left(p(\lambda_{p,j}-(Vu_{p,j},u_{p,j}))\right)^{1/2}\leq C_1\sqrt{p},
\end{equation}
Since $\lambda_{p,j}\leq \beta$ and the functions $\nabla\chi$ and $\Delta\chi$ are bounded on $X$, by \eqref{e:comm1} and \eqref{e:Hvjvj2}, we infer that
\begin{equation}\label{e:comm2}
\|[H_p, \chi]u_{p,j}\|=\frac {C_2}{\sqrt{p}}.
\end{equation}
Using \eqref{e:Hvjvj0}, \eqref{e:comm2}, $\lambda_{p,j}\leq \beta$ and $\|v_{p,j}\|\leq 1$, we infer that 
\begin{equation}\label{e:Hvjvj1}
|(H_pv_{p,j},v_{p,j})|\leq \gamma, \quad p\in \mathbb N, \quad j=1,2,\ldots, N_p,
\end{equation}
with $\gamma>0$ independent of $p$ and $j$. 

For $j\neq k$, since $(H_pu_{p,j},u_{p,k})=0$, we can write
\[
(H_pv_{p,j},v_{p,k})
=-2(H_pv_{p,j},(1-\chi)u_{p,k})-(H_p[(1-\chi)u_{p,j}],(1-\chi)u_{p,k}).
\]
By \eqref{e:exp1}, we have
\begin{equation}\label{e:1-chi}
\|(1-\chi)u_{p,k}\|=\mathcal O(e^{-cr\sqrt{p}}).
\end{equation}
By \eqref{e:comm2}, \eqref{e:1-chi}, $\lambda_{p,j}\leq \beta$, we get
\[
\|H_pv_{p,j}\|\leq \|[H_p,\chi] u_{p,j}\|+\lambda_{p,j}\|v_{p,j}\|\leq C_3
\]
and
\[
\|H_p[(1-\chi)u_{p,j}]\|\leq \|[H_p,\chi] u_{p,j}\|+\lambda_{p,j}\|(1-\chi)u_{p,j}\|\leq \frac{C_4}{\sqrt{p}}.
\]
This shows that 
\begin{equation}\label{e:Hvjvk}
|(H_pv_{p,j},v_{p,k})|\leq C_5e^{-cr\sqrt{p}}, \quad p\in \mathbb N, \quad j,k=1,2,\ldots, N_p,\quad j\neq k.
\end{equation}

Let $\mathcal V_p\subset C^\infty(X,L^p)$ be the linear span of the set $\{v_{p,j}, j=1,2,\ldots, N_p\}$. 
By \eqref{e:vjvj}, \eqref{e:vjvj1} and \eqref{e:vjvk}, the set $\{v_{p,j}, j=1,2,\ldots, N_p\}$ is linearly independent and, therefore, $\dim \mathcal V_p=N_p$ for sufficiently large $p$. 
Take an arbitrary $v\in \mathcal V_p$ of the form
\[
v=\sum_{j=1}^{N_p}c_jv_{p,j}.
\]
By \eqref{e:vjvj1} and \eqref{e:vjvk}, we have
\begin{equation}\label{e:v2}
\|v\|^2=\sum_{j,k=1}^{N_p}c_j\overline{c_k}(v_{p,j},v_{p,k})\geq \sum_{j=1}^{N_p}|c_j|^2(1-C_1e^{-2cr\sqrt{p}})
\end{equation}
with some $C_1>0$ independent of $p$ and $v$.

Next, by \eqref{e:Hvjvj1}and \eqref{e:Hvjvk}, we infer that
\begin{equation}\label{e:Hpvv}
(H_pv,v)= \sum_{j=1}^{N_p}c_j\overline{c_k}(H_pv_{p,j},v_{p,k})\leq \gamma \sum_{j=1}^{N_p}|c_j|^2(1+C_2e^{-2cr\sqrt{p}})
\end{equation}
with some $C_2>0$ independent of $p$ and $v$.

Let $p_1\in \mathbb N$ be such that for any $p>p_1$ we have $C_1e^{-2cr\sqrt{p}}<1/2$ and, therefore, $1-C_1e^{-2cr\sqrt{p}}>1/2$. By \eqref{e:v2} and \eqref{e:Hpvv}, for any $p>p_1$ we obtain
\begin{multline}\label{e:Hpvv1}
(H_pv,v)\leq \gamma \sum_{j=1}^{N_p}|c_j|^2(1+C_2e^{-2cr\sqrt{p}})
\leq \gamma \left(1+\frac{C_2}{2C_1}\right) \sum_{j=1}^{N_p}|c_j|^2\\
 \leq \gamma \left(2+\frac{C_2}{C_1}\right) \sum_{j=1}^{N_p}|c_j|^2(1-C_1e^{-2cr\sqrt{p}}) \leq \mu \|v\|^2
\end{multline}
with $$\mu=\gamma \left(2+\frac{C_2}{C_1}\right).$$ 

Let us consider the Dirichlet realization $H_{p,D}$ of the operator $H_p$ in $D$. It is clear that $\mathcal V_p$ is contained in the domain of $H_{p,D}$. Using \eqref{e:Hpvv} and the variational principle (Glazman's lemma), we conclude that 
\begin{equation}\label{e:NN1}
N_p=N([\alpha,\beta];H_p)\leq N(\mu; H_{p,D}), \quad p>p_1.
\end{equation}

\begin{lem}
For any $\mu\in \mathbb R$, there exist $C>0$ and $p_2\in \mathbb N$ such that
\begin{equation}\label{e:NN2}
N(\mu; H_{p,D})\leq Cp^{-n}, \quad p>p_2.
\end{equation}
\end{lem}

\begin{proof}
This lemma is essentially Lemma 4.2 in \cite{HM96} with the only difference that we consider operators acting on vector bundles. The proof is almost the same because each vector bundle is locally trivial. 
\end{proof}

Combining \eqref{e:NN1} and \eqref{e:NN2}, we complete the proof of Proposition~\ref{p:upperbound}. 
\end{proof}

\end{document}